\documentclass[12pt]{article}
\usepackage{amssymb,amsfonts,amsmath,amsthm}
\newcommand{\ol}{\overline}
\newcommand{\one}[1]{\mbox {\bf 1}_{\{#1\}}}
\newcommand{\ms}{{\mathfrak S}}

\newcommand{\witi}{\widetilde}

\newcommand{\nn}{{\mathbb N}}
\newcommand{\rr}{{\mathbb R}}
\newcommand{\zz}{{\mathbb Z}}
\newcommand{\cald}{{\mathcal D}}
\newcommand{\call}{{\mathcal L}}
\newcommand{\calr}{{\mathcal R}}
\newcommand{\cals}{{\mathcal S}}
\newcommand{\calt}{{\mathcal T}}
\newcommand{\veps}{\varepsilon}

\newcommand{\beq}{\begin{eqnarray*}}
\newcommand{\feq}{\end{eqnarray*}}
\newcommand{\beqn}{\begin{eqnarray}}
\newcommand{\feqn}{\end{eqnarray}}
\newcommand{\as}{\mbox{\rm a.\,s.}}
\newcommand{\io}{\mbox{\rm i.\,o.}}
\newtheorem{theorem}{Theorem}
\makeatletter \@addtoreset{theorem}{section}\makeatother

\newtheorem{lemma}[theorem]{Lemma}
\newtheorem{assume}[theorem]{Assumption}
\newtheorem*{theorem*}{Theorem}

\newtheorem{remark}[theorem]{Remark}

\title{On a directionally reinforced random walk}
\author{
Arka Ghosh\footnote{Department of Statistics, Iowa State University, Ames, IA 50011, USA; e-mail: apghosh@iastate.edu}
\and
Reza Rastegar\footnote{Department of Mathematics, Iowa State University, Ames, IA 50011, USA; e-mail: rastegar@iastate.edu}
\and
Alexander Roitershtein\footnote{Department of Mathematics, Iowa State University, Ames, IA 50011, USA; e-mail: roiterst@iastate.edu} }
\begin{document}
\maketitle
\begin{abstract}
We consider a generalized version of a directionally reinforced random walk, which was originally introduced by Mauldin, Monticino,
and von Weizs\"{a}cker in \cite{drw}. Our main result is a stable limit theorem for the position of the random walk in higher dimensions.
This extends a result of Horv\'{a}th and Shao \cite{limits} that was previously obtained in dimension one only (however,
in a more stringent functional form).
\end{abstract}
{\em MSC2000: } Primary: 60F15, 60F17, 60F20; Secondary:  60J25, 70B05.\\
\noindent {\em Keywords:} correlated random walks, directional reinforcement, limit theorems, stable laws.

\section{Introduction}
\label{intro}
In this paper we study the following {\em directionally reinforced random walk}.
Fix $d\in \nn$ and a finite set $U$ of distinct unit vectors in $\rr^d$ (see
Remark~\ref{extensions} at the end of Section~\ref{results} below, where a
suitable alternative Markovian setup in a general state space is discussed).
The vectors in $U$ serve as feasible directions for the motion of the random walk.
To avoid trivialities we assume that $U$ contains at least two elements. Let $X_t \in \rr^d$
denote the position of the random walk at time $t.$ Throughout the paper we assume that $X_0=0.$
The random walk changes its direction at random times
\beq
s_1:=0<s_2<s_3<s_4<....
\feq
We assume that the time intervals
\beq
T_n:=s_{n+1}-s_n,\qquad n\in\nn,
\feq
are independent and identically distributed. Let $\eta_n\in U$ be the direction of the walk during
time interval $[s_n,s_{n+1}).$ We assume that $\eta:=(\eta_n)_{n\geq 1}$ is an irreducible stationary Markov chain on $U$ which is,
furthermore, independent of $(s_n)_{n\in\nn}.$
\par
For $t>0,$ let $N_t:=\sup\bigl\{k\geq 1: s_k\leq t \bigr\}$ be the number of times that the walker changes direction before time $t>0.$
Then
\beqn
\label{rw}
X_t= \sum_{i=1}^{N_t-1}{\eta_i T_i} + (t-s_{N_t})\eta_{N_t} .
\feqn
Notice that $N_t\geq 1$ with probability one, due to the convention $s_1=0$ that we have made.
\par
The random walk $X_t$ defined above is essentially the model introduced by Mauldin, Monticino,
and von Weizs\"{a}cker in \cite{drw} and further studied by Horv\'{a}th and Shao in \cite{limits} and
by Siegmund-Schultzea and von~Weizs\"{a}cker in \cite{levelcr}. The technical difference between our model and the variant which has been studied
in \cite{limits} is that in the latter, the next direction of the motion is chosen uniformly from the available
set of ``fresh directions", while we do not impose any restrictions on the transition kernel of $\eta$ besides irreducibility.
\par
The original model proposed in \cite{drw} was inspired by certain phenomena that occur in ocean surface waves
(cf. \cite{West}) and was designed to reproduce the same features within a probabilistic framework.
The main topic of \cite{drw} and \cite{levelcr} is recurrence-transience criteria. Horv\'{a}th and Shao in \cite{limits} studied scaling limits of the random walk
in different regimes, answering some of the questions which have been posed in \cite{drw}.
\par
We remark that somewhat related random walk models have been considered by Allaart and Monticino in
\cite{srule,rules1} and by Gruber and Schweizer in \cite{diffusion}.
In the context of random walks in random environment, a similar in spirit model of {\em persistent} random walks
was introduced by Sz\'{a}sz and T\'{o}th in \cite{persian,semak}. The common feature of ``generic versions" of all models
mentioned above is that the underlying random motion has a tendency to persist in its current direction.
\par
Closely related to persistent random walks are recurrent ``random flights" models where changes of the direction of the random motion
follow a Poisson random clock. These models can be traced back to Pearson's random walk \cite{drct8, drct9} and Goldstein-Kac
one-dimensional ``telegraph process" \cite{drct7, drct6} .
Random flights have been intensively studied since the introduction of the telegraph process
in the early 50's, see for instance \cite{flights,drct5,drct,drct1,drct3,drct4} and references therein for a representative sample.
An introductory part of \cite{drct} provides a short authoritative and up to a date survey of the field.
We remark that, somewhat in contrary to directionally reinforced random walks, the main focus of the research in this area
is on finding explicit form of limiting distributions for these processes.
\par
The main goal of this paper is to prove stable limit theorems for the directionally reinforced random walk
in arbitrary dimension $d\geq 1.$ In addition, we extend some limit results of \cite{limits} to our setting and
also complement them by suitable laws of iterated logarithm. Our proofs can be easily carried over to a setup where the set of feasible directions
$U$ is not finite, but is rather supported (under the stationary law of the process) on a general Borel subset of the
unit sphere; see Remark~\ref{extensions} below for more details.
\par
Our results are stated in Section \ref{results} whereas the proofs are contained in Section \ref{proofs}.
The non-Gaussian limit theorems for the position of the random walk in higher dimensions, stated in
Theorems~\ref{clts} and \ref{clt1} constitute the main contribution of this paper.
\section{Statement of main results}
\label{results}
We first introduce a few notations.
For a vector $x=(x_1,\ldots,x_d)\in\rr^d$ let $\|x\|=\max_i |x_i|.$ For (possibly, random) functions $f,g:\rr_+~(\mbox{or}~\nn)~\to\rr,$
write $f\sim g$ and $f(t)=o(g(t))$ to indicate that, respectively, $\lim_{t\to\infty} f(t)/g(t)=1$ and $\lim_{t\to\infty}{f(t)/g(t)}=0,$ $\as$
Let $\pi=(\pi_v)_{v\in U}\in \rr^{|U|}$ be the unique stationary distribution of the Markov chain $\eta$ and let
\beqn
\label{demu}
\mu=\sum_{v\in U}\pi_v v.
\feqn
Thus $\mu=E(\eta_n)\in\rr^d$ for each $n\in\nn.$
\par
The following theorem shows that a strong law of large numbers holds for $X_t$ and that, under suitable second moment condition,
the sample paths of the random walk are uniformly close to the sample paths of a drifted Brownian motion.
We have:
\begin{theorem}
\label{strapprox}
\item [(a)] Suppose that $E(T_1^p)<\infty$ for some constant $p \in (1,2).$ Then,
\beq
\big \|X_t- \mu t\big\| = o\bigl(t^{1/p}\bigr).
\feq
\item [(b)]
If $E(T_1^p) < \infty$ for some constant $p>2,$ then (in an enlarged, if needed, probability space) there exist a process $\widehat X=\bigl(\widehat X_t)_{t\geq 0}$
distributed as $X$ and a Brownian motion $(W_t)_{t\geq 0}$ in $\rr^d,$ such that,
\beq
\sup_{0\leq t\leq T}\bigl \| \widehat X_t -  \mu t-W_t \bigr\| = o\bigl(T^{1/p}\bigr).
\feq
\end{theorem}
\begin{remark}
The results stated in Theorem~\ref{strapprox} as well as in Theorem~\ref{clt} below are essentially due to \cite{limits}.
In fact, the original proofs can be adapted to our more general setup.
However, the proofs we give in Section~\ref{proofs} are shorter and somewhat simpler than the original ones.
Furthermore, our proves can easily be seen working for the general Markov chain setup described in Remark~\ref{extensions} below.
\end{remark}
The second part of Theorem~\ref{strapprox} implies the invariance principle for $(X_{nt}-\mu n t)$
with the usual normalization $\sqrt{n}.$ We next state an invariance principle and the corresponding law of iterated
logarithm under a slightly more relaxed moment condition.
Let $D\bigl(\rr^d\bigr)$ denote the set of $\rr^d$-valued c\`{a}dl\`{a}g functions
on $[0,1]$ equipped with the Skorokhod $J_1$-topology. We use notation $\Rightarrow$ to denote the weak convergence in $D\bigl(\rr^d\bigr).$
We have:
\begin{theorem}
\label{clt}
For $n\in \nn,$ define a process $S_n$ in $D\bigl(\rr^d\bigr)$ by setting
\beqn
\label{dpath1}
S_n(t)=\frac{X_{nt}-\mu n t}{\sqrt{n}}, \qquad t \in [0,1].
\feqn
If $E(T_1^2)<\infty,$ then
\item [(a)] $S_n\Rightarrow W,$ where $W=(W_t)_{t\geq 0}$ is a (possibly degenerate, but not identically equal to zero)
$d$-dimensional Brownian motion.
\item [(b)] For every $x\in\mbox{\rm Span}(U)\subset \rr^d,$ there is a constant $K(x)\in (0,\infty)$ such that
\beq
\limsup_{t\to\infty} \frac{\left(X_t-\mu t\right)\cdot x}{\sqrt{t\ln\ln t}} = K(x).
\feq
Furthermore, a similar statement holds for the $\liminf.$
\end{theorem}
We next consider the case when $E(T_1^2)=\infty$ and $T_1$ is in the domain of attraction of a stable law.
Namely, for the rest of our results we impose the following assumption.
Recall that a function $h:\rr_+\to\rr$ is said to be {\em regularly varying of index} $\alpha\in\rr$
if $h(t)=t^\alpha L(t)$ for some $L:\rr_+\to\rr$ such that $L(\lambda t)\sim L(t)$ for all $\lambda >0.$
We will denote the set of all regularly varying functions of index $\alpha$ by $\calr_\alpha.$
\begin{assume}
\label{amain}
There is $h\in\calr_\alpha$ with $\alpha \in (0,2]$ such that $\lim_{t\to\infty} h(t) \cdot P(T_1>t)\in (0,\infty).$
\end{assume}
For $t>0$ let
\beqn
\label{an}
a_t=
\left\{
\begin{array}{lll}
\inf\,\{s>0: t\cdot P(T_1>s)\leq 1\}&\mbox{\rm if}&\alpha <2,\\
&
\\
\inf\,\{s>0: ts^{-2}\cdot E\bigl(T_1^2;\,T_1\leq s\bigr)\leq 1\}&\mbox{\rm if}&\alpha =2
\end{array}
\right.
\feqn
If $h(t)\in\calr_\alpha$ with $\alpha\in(1,2]$ (and hence $E(T_1)<\infty$), one can obtain the following
analogue of Theorem~\ref{clt}. It turns out that also in this case the functional limit theorem and
the law of iterated logarithm for $X_t$ inherit the structure of the corresponding
statements for the partial sums of i.i.d. variables $\sum_{k=1}^n T_k.$
\begin{theorem}
\label{clts}
Let Assumption~\ref{amain} hold with $\alpha\in (1,2].$ Let
\beq
S_t:=\frac{X_t-\mu t}{a_t}, \qquad t>0.
\feq
We have:
\item [(a)] If $\alpha \in (1,2),$ then
\begin{itemize}
\item [(i)] $S_t$ converges weakly to a non-degenerate multivariate stable law in $\rr^d.$
\item [(ii)] For every $x\in\mbox{\rm Span}(U) \subset \rr^d$ such that $x\cdot u>0$ for some $u\in U,$
\beq
\limsup_{t\to\infty} \frac{(X_t-\mu t)\cdot x}{a_t\cdot (\ln t)^{1/\alpha+\veps}}=
\left\{
\begin{array}{ll}
0&\mbox{\rm ~if~}~\veps>0,\\
\infty&\mbox{\rm ~if~}~\veps<0
\end{array}
\right.\qquad \mbox{\rm a.\,s.}
\feq
In particular, for some constant $c(x)>0,$
\beq
\limsup_{t\to\infty} \Bigl\{\frac{(X_t-\mu t) \cdot x}{a_t}\Bigr\}^{1/\ln\ln t}=c(x)\qquad \mbox{\rm a.\,s.}
\feq
\end{itemize}
\item [(b)] If $\alpha=2$ and $E(T_1^2)=\infty,$ then $S_t$ converges weakly to a non-degenerate multivariate Gaussian
distribution in $\rr^d.$
\end{theorem}
For $\alpha \in (0,1)$ we have the following limit theorem.
\begin{theorem}
\label{clt1}
Let Assumption~\ref{amain} hold with $\alpha\in (0,1).$ Then $\frac{X_t}{t}$ converges weakly in $\rr^d$ to a non-degenerate limit.
\end{theorem}
\begin{remark}
The limiting random law in the statement of Theorem~\ref{clt1} is specified in
\eqref{specific} below. The stable limit laws for $X_t$ stated in Theorems~\ref{clts} and Theorem~\ref{clt1} are extensions of corresponding
one-dimensional results in \cite{limits}. The latter however are obtained in \cite{limits} in a more stringent
functional form. The law of iterated logarithm given in Theorem~\ref{clts}
appears to be new even for $d=1.$
\end{remark}
\begin{remark}
\label{extensions}
Recall  Markov chain $\eta=(\eta_n)_{n\geq 0}$ which records successive directions of the random walk.
Let $\cals^{d-1}$ denote the $d$-dimensional unit sphere and let $\calt_d$ denote the $\sigma$-algebra of
the Borel sets of $\cals^{d-1}.$ Denote by $H(x,A)$ transition kernel of $\eta$ on $(\cals^{d-1},\calt_d).$
We remark that
\item [(i)] All the results stated in this section remain true for an arbitrary (not stationary)
initial distribution of the Markov chain $\eta.$
\item [(ii)] The proofs of our results given in Section~\ref{proofs} rest on the exploiting
of a regenerative (renewal) structure associated with $\eta,$ i.e. on the use of random times $\tau_n$ which are
introduced below in Section~\ref{prelim}. It is then not hard to verify that all the results stated in this section, with the only exception of the generalized
law of iterated logarithm given in part (a)-(ii) of Theorem~\ref{clts}, remain true
for a class of regenerative (in the sense of \cite{renewal-appr}) Markov chains $\eta$
whose stationary  distribution are supported on general Borel subsets of $\cals^{d-1}$
rather than on a finite set $U\subset \cals^{d-1}.$ For instance,
the following strong version of the classical Doeblin's conditions is sufficient for our purposes:
\begin{itemize}
\item  There exist a constant $c_r > 1$ and a probability
measure $\psi$ on $(\cals^{d-1},\calt_d)$ such that
\beqn
\label{bounds}
c^{-1}_r\psi(A) < H(x,A)<c_r \psi(A)\qquad
\forall x \in \cals,\,A \in \calt_d.
\feqn
\end{itemize}
A regenerative (renewal) structure for Markov chains which satisfies Doeblin's condition is described in \cite{renewal-appr}.
Due to the fact that under the assumption \eqref{bounds}, the kernel $H(x,A)$ is dominated uniformly from above and below by
a probability measure $\psi,$ such Markov chains share two key features with finite-state Markov chains. Namely,
1) the exponential bound stated in \eqref{tau-dist} holds for the renewal times which are defined in \cite{renewal-appr};
and 2) $c^{-1}_r < P_x(A)/P_y(A)<c_r$ for any non-null event $A \in \calt_d$ and almost every states $x,y\in \cals^{d-1}$
(with respect to the stationary law). Here $P_x$ stands for the law of the Markov chain $\eta$ starting from the initial state $x\in \cals^{d-1}.$
Once these two crucial properties are verified, our proofs (except only the proof of part (a)-(ii) of Theorem~\ref{clts})
work nearly verbatim for directionally reinforced random walks governed by a Markov chain $\eta$ which satisfies
condition \eqref{bounds}.
\end{remark}

\section{Proofs}
\label{proofs}
This section is devoted to the proof of the results stated in Section~\ref{results} above.
Some preliminary observations are stated in Section~\ref{prelim} below.
The proof of Theorem~\ref{strapprox} is contained in Section~\ref{proof-strapprox}.
Theorems~\ref{clt} and \ref{clts} are proved in Section~\ref{proof-clt} and Section~\ref{proof-clts}, respectively.
Finally, the proof of Theorem~\ref{clt1} is given in Section~\ref{proof-clt1}.
\subsection{Preliminaries}
\label{prelim}
Our approach relies on the use of a renewal structure which is
induced on the paths of the random walk by the cycles of the underlying Markov chain $\eta.$
To define the renewal structure, set $\tau_0=0$ and let
\beq
\tau_{i+1}=\inf\bigl\{ j>\tau_i : \eta_j= u_1 \bigr\}, \qquad i\geq 0.
\feq
Thus, for $i\geq 1,$ $\tau_i$ are steps when the Markov chain $\eta$ visits the distinguished
state $u_1.$ Correspondingly, $s_{\tau_i}$ are successive times when the random walk chooses $u_1$ as the direction
of its motion. Recall $N_t$ from Section~\ref{intro} (see a few lines preceding \eqref{rw}).
Denote by $c(t)$ the number of times that the walker chooses direction $u_1$ before time $t>0.$
That is,
\beq
c(t):= \sup\bigl\{i\geq 0: s_{\tau_i}\leq t \bigr\}=\sum_{j=1}^{N_t} \textbf {1}_{\{\eta_j=u_1\}},
\feq
where $\textbf{1}_A$ stands for the indicator function of an event $A.$ Notice that $N_t$
is the unique mapping from $\rr_+$ to $\zz_+$ which has the following property:
\beq
s_{N_t}\leq t < s_{N_t+1}\qquad \mbox{and} \qquad \tau_{c(t)}\leq N_t < \tau_{c(t)+1}.
\feq
For $i\geq 0,$ let $\xi_i=\sum_{j=\tau_i+1}^{\tau_{i+1}}{T_j\eta_j}.$
Then
\beqn \label{decomp1}
X_t= \xi_0+\sum_{i=1}^{c(t)-1} \xi_i+\sum_{j=\tau_{c(t)}+1}^{N_t} T_j\eta_j + \bigl(t-s_{N_t}\bigr)\cdot \eta_{N_t}.
\feqn
The strong Markov property implies that the pairs $\bigl(\xi_i,\tau_{i+1}-\tau_i\bigr)_{i\in\nn}$ form an i.i.d. sequence
which is independent of $(\xi_0,\tau_1)$.
Furthermore, since $\eta$ is an irreducible finite-state Markov chain, there exist positive constants $K_1,K_2>0$
such that the inequality
\beqn
\label{tau-dist}
P(\tau_{i+1}-\tau_i > t) \leq K_1e^{-K_2\,t}
\feqn
holds uniformly for all reals $t\geq 0$ and all integers $i\geq 0.$
\par
We next list some direct consequences of the law of large numbers that will be frequently exploited
in the subsequent proofs. Let $v(n)$ be the number of times that the Markov chain $\eta$ visits $u_1$ during its first $n$ steps.
Thus, while $c(t)$ is the number of visits of $\eta$ to $u_1$ up to time $t>0$ on the clock of the random walk,
$v(n)$ is the number of occurrences of $u_1$ among first $n$ directions of the random walk. In particular,
$v(N_t)=c(t).$ Taking into account \eqref{tau-dist}, the law of large numbers and the renewal theorem imply that
\beq
\lim_{n\to\infty}\frac{\tau_n}{n}=\lim_{n\to\infty}\frac{n}{v(n)}=E(\tau_2-\tau_1)=\pi_1^{-1},\qquad \as,
\feq
and, letting $\Lambda_k:=\sum_{i=\tau_k+1}^{\tau_{k+1}} \eta_i,$
\beq
\mu &=& \lim_{n\to\infty} \frac{\sum_{i=1}^n\eta_i}{n} = \lim_{n\to\infty}\frac{\sum_{k=0}^{v(n)}\Lambda_k }{n} =
\pi_1\cdot E(\Lambda_1), \qquad \as
\feq
Since $\eta$ and $(T_k)_{k\in\nn}$ are independent, it follows that
\beqn
\label{average}
E(\xi_1)= E(T_1)\cdot E(\Lambda_1)=\pi_1^{-1} \mu \cdot E(T_1).
\feqn
Finally, $\frac{c(t)}{t}=\frac{v(N_t)}{t}= \frac{v(N_t)}{N_t}\cdot \frac{N_t}{t}$ yields
\beqn
\label{malpha}
 \lim_{t\to\infty}\frac{c(t)}{t}= \frac{\pi_1}{E(T_1)},\qquad \as
\feqn
We now turn to the proofs of our main results.
\subsection{Proof of Theorem~\ref{strapprox}}
\label{proof-strapprox}
\begin{proof} [\textbf{Part (a) of Theorem~\ref{strapprox}}]
Recall \eqref{tau-dist} and observe that the moment condition of the theorem along with the independence of the Markov chain $\eta$ and $(T_k)_{k\in\nn}$
of each other, implies that
\beqn
\nonumber
E(\| \xi_1 \|^p) &\leq& E\bigl[\bigl( s_{\tau_2}-s_{\tau_1}\bigr)^p\bigr] = \sum_{n=1}^\infty P(\tau_2-\tau_1=n) \cdot E\Bigl[\Bigl(\sum_{k=1}^n T_k\Bigr)^p\Bigr]
\\
\label{moments}
&\leq& K_1 \sum_{n=1}^\infty e^{-K_2 (n-1)} n^p E(T_1^p) <\infty,
\feqn
where we used Minkowski's inequality and \eqref{tau-dist}. It follows that $\|\xi_k\|=\mbox{o}\bigl(k^{1/p}\bigr).$
Indeed, for any $\veps>0,$ Chebyshev's inequality implies that
\beq
\sum_{k=1}^\infty P\bigl(\|\xi_k\|>k^{\frac{1}{p}}\veps \bigr)= \sum_{k=1}^\infty P\bigl(\|\xi_k\|^{p}>\veps^{p}k \bigr)
\leq \veps^{-p} E\bigl(\|\xi_1\|^{p}\bigr)<\infty,
\feq
and hence $P\bigl(\|\xi_k\|>k^{\frac{1}{p}}\veps~ \mbox{\io}\bigr)=0$ by the Borel-Cantelli lemma.
\par
For now we will make a simplifying assumption (to be removed later on) that $\mu=0.$
By virtue of \eqref{malpha}, the Marcinkiewicz-Zigmund law of large numbers implies that
\beq
\lim_{t\to\infty}\frac{\sum_{i=0}^{c(t)-1}{\xi_i}}{t^{1/p}}=0,\qquad \as
\feq
Furthermore, by \eqref{decomp1}, $\bigl \|X_t-\sum_{i=0}^{c(t)-1}\xi_i\bigr \|  \leq r_{c(t)},$
where
\beqn
\label{rk}
r_k:= \sum_{i=\tau_{k}+1}^{\tau_{k+1}}T_i.
\feqn
An argument similar to the one which we used to estimate the order of $\|\xi_n\|,$ shows that with probability one $r_n=o(n^{1/p}).$
Then \eqref{malpha} implies that
\beqn
\label{momentsr}
r_{c(t)}=o(t^{1/p}).
\feqn
This completes the proof of part~(a) of Theorem~\ref{strapprox} for the particular case $\mu=0.$
\par
We now turn to the general case of arbitrary finite $\mu\in\rr^d.$  Let
\beqn
\label{tilde}
\tilde{\eta}_i=\eta_i-\mu \qquad \mbox{and}\qquad \witi X_{t}=\sum_{i=0}^{N_t}{T_i\tilde{\eta}_i} + (t-s_{N_t})\tilde{\eta}_{N_t}.
\feqn
Then $\witi X_t$ is a directionally reinforced random walk associated with $(T_n)_{n\in\nn}$ and $\tilde \eta=(\tilde \eta_n)_{n\in \nn}.$
Since $E(\tilde \eta_i)=0,$ we have $\bigl\| \witi X_t \bigr \| = o\bigl(t^{1/p}\bigr).$
To complete the proof of part~(a) of the theorem, observe that $X_t - \witi X _t=\mu\cdot \sum_{i=1}^{N_t}{T_i} + \mu\cdot (t-s_{N_t}) = \mu t.$
\end{proof}
$\mbox{}$
\begin{proof} [\textbf{Part (b) of Theorem~\ref{strapprox}}]
Recall \eqref{average}. Let
\beq
\bar{\xi}_k := \xi_k - E(\xi_1) =  \xi_k - E(T_1) \pi_1^{-1}\mu
\feq
and
\beq
\Delta_k:= s_{\tau_{k+1}-1}-s_{\tau_k-1}-E(T_1) \pi_1^{-1}.
\feq
Let $\gamma_k=(\bar{\xi}_k, \Delta_k)\in \rr^{d+1}.$  Then $(\gamma_k)_{k\geq 1}$ is an i.i.d. sequence with $E(\gamma_1)=0\in\rr^{d+1}.$
Define
\beq
\Gamma(t)=\sum_{1\leq k\leq t}\gamma_k.
\feq
By virtue of Theorem~1~2.~1 in \cite{approxim}, there is a Brownian motion $\bigl(B(t)\bigr)_{t\geq 0}$ in $\rr^{d+1}$ such that
\beq
\sup_{0\leq t\leq T} \bigl \|\Gamma(t)-B(t) \bigr\| =o\bigl(T^{1/p}\bigr).
\feq
Then Theorem~2.~3.~6 in \cite{approxim} implies that there exists a Brownian motion $\bigl(W(t)\bigr)_{t\geq 0},$ such that
\beq
\sup_{0\leq t\leq T}\Bigl\|\sum_{k=0}^{c(t)-1} \xi_k -  t\mu   - W(t) \Bigr\| = o\bigl(T^{1/p}\bigr).
\feq
Recall $r_k$ from \eqref{rk}. Since
\beq
\sup_{0\leq t \leq T} \Bigl\| X_t- \sum_{k=0}^{c(t)-1}\xi_k\Bigr\| \leq \sup_{0\leq t \leq T} r_{c(t)},
\feq
it suffices to show that
\beqn
\label{maxe}
\sup_{0\leq t \leq T} r_{c(t)} = o(T^{1/p}).
\feqn
Notice that
\beq
\sup_{0\leq t \leq T} r_{c(t)}=\sup_{0\leq k \leq c(T)} r_k.
\feq
Therefore, by virtue of \eqref{malpha}, it suffices to show that
\beqn
\label{tend}
\lim_{n\to\infty} n^{-1/p} \cdot \sup_{0\leq k \leq n} r_k=0,\qquad \as
\feqn
Toward this end, let
\beq
g(n)=\max\bigl\{k\leq n: r_k\geq r_i~\mbox{for all}~1\leq i\leq n\bigr\},\qquad n\in\nn.
\feq
Thus $g(n)\leq n$ and $\sup_{0\leq k \leq n} r_k = r_{g(n)}.$ Furthermore, since $r_k$ are i.i.d. random
variables, $\lim_{n\to\infty}g(n)= \infty$ with probability one. Therefore, $r_n=o(n^{1/p})$
yields \eqref{tend}. The proof of Theorem~\ref{strapprox} is completed.
\end{proof}
\subsection{Proof of Theorem~\ref{clt}}
\label{proof-clt}
\begin{proof} [\textbf{Part (a) of Theorem~\ref{clt}}]
By \eqref{moments}, $E\bigl(\|\xi_1\|^2)\bigr)<\infty$ under the conditions of the theorem. Assume first that
$\mu=0.$ Then the invariance principle for i.i.d. sequences implies that
\beq
\frac{\sum_{k=1}^{[nt]} \xi_k }{\sqrt{n}} \Rightarrow W(t),\qquad t \in [0,1],
\feq
where $W(t)$ is a $d$-dimensional Brownian motion. It follows then from \eqref{malpha}
and Theorem~14.4 in \cite[p.~152]{measures} that
\beqn
\label{general}
\frac{\sum_{k=0}^{c(nt)-1} \xi_k }{\sqrt{n}} \Rightarrow  \sqrt{b} \cdot W(t),\qquad t \in [0,1],
\feqn
where $b=\frac{\pi_1}{\cdot E(T_1)}.$
Under the moment condition of Theorem~\ref{clt} we have the following counterpart of \eqref{maxe}:
\beq
\sup_{0\leq t \leq T} r_{c(t)} = o(T^{1/2}).
\feq
Since $\bigl\|X_{nt}-\sum_{k=0}^{c(nt)-1} \xi_k\bigr\| $ is
bounded above by $r_{c(nt)},$ it follows that
\beq
n^{-1/2}\cdot \Bigl\|X_{nt}-\sum_{k=0}^{c(nt)-1} \xi_k\Bigr\|  \Rightarrow 0,
\feq
which implies the desired convergence of $ n^{-1/2}\cdot X_{nt}$ when $\mu=0.$
To prove the general case of arbitrary $\mu\in\rr^d$ one can apply the result
with $\mu=0$ to the Markov chain $\tilde \eta_n$ and the random walk $\tilde X_t$
that were introduced in \eqref{tilde}. The proof of part~(a) of the theorem is completed.
\end{proof}
$\mbox{}$
\begin{proof} [\textbf{Part (b) of Theorem~\ref{clt}}]
Suppose first that $\mu=0.$ For $x\in\mbox{Span}(U) \subset \rr^d$  and $i\in\nn$ define
\beqn
\label{lileq1}
\xi_{i,x} := \xi_i\cdot x.
\feqn
Then, in view of \eqref{malpha}, the law of iterated logarithm for i.i.d. sequences implies that
there exists a constant $K(x)\in (0,\infty)$ such that
\beq
\limsup_{t\to\infty} \frac{\sum_{i=0}^{c(t)-1}\xi_{i,x}}{\sqrt{t\ln\ln t}}=K(x), \qquad \as
\feq
By \eqref{average} and \eqref{momentsr}
\beq
\lim_{t\to\infty}  \frac{\bigl|X_t\cdot x-\sum_{i=0}^{c(t)-1}\xi_{i,x}\bigr|}{\sqrt{t\ln\ln t}}=0,\qquad \as
\feq
Thus
\beq
\limsup_{t\to\infty} \frac{X_t \cdot x}{\sqrt{t\ln\ln t}}=K(x), \qquad \as,
\feq
in the case $\mu=0.$ To obtain the general case with an arbitrary $\mu\in\rr^d,$ apply this result to
the random walk $\witi X_t$ defined in \eqref{tilde} and recall that $X_t-\witi X_t=\mu t.$
The proof of part~(b) of Theorem~\ref{clt} is completed.
\end{proof}
\subsection{Proof of Theorem~\ref{clts}}
\label{proof-clts}
\begin{proof} [\textbf{Part (a)-(i) and part~(b) of Theorem~\ref{clts}}]
Let $\ol \rr_0^d:=[-\infty,\infty]^d\backslash\{0\},$
where $0$ stands for the zero vector in $\rr^d,$ and equip $\ol \rr_0^d$ with the topology inherited from $\rr^d.$
Recall (see for instance \cite{character,foundations}) that a random vector $\xi \in\rr^d$ is said to be regularly varying with index $\alpha> 0$
if there exists a function $a:\rr_+\to\rr,$ regularly varying with index $1/\alpha,$ and a Radon measure $\nu_\xi$ on $\ol \rr_0^d$ such that
\beqn
\label{nprop}
n P\bigl(a_n^{-1}\xi \in \cdot\bigr) \overset{v}{\Rightarrow}\nu_\xi(\cdot), \qquad \mbox{as}~n\to\infty,
\feqn
where $ \overset{v}{\Rightarrow}$ denotes the vague convergence of measures.
We will denote by $\calr_{d,\alpha,a}$ the set of all random $d$-vectors regularly varying with index $\alpha,$
associated with a given function $a\in\calr_{1/\alpha}$ by \eqref{nprop}. The measure $\nu$ is referred to as the {\em measure of regular variation} associated with $\xi.$
We will also use the following equivalent definition of the regular variation for random vectors (see, for instance, \cite{character,foundations}).
Let $S^{d-1}$ denote the unit sphere in $\rr^d$ with respect to the norm $\|\cdot\|.$ Then $\xi \in \calr_{d,\alpha,a}$ if and only if there exists
a finite Borel measure $\ms_\xi$ on $S^{d-1}$ such that for all $t>0,$
\beqn
\label{dprop}
n P\bigl(\|\xi\|>ta_n;\,\xi/\|\xi\|\in \cdot\bigr) \overset{v}{\Rightarrow}
t^{-\alpha}\ms_\xi(\cdot), \qquad \mbox{as}~n\to\infty,
\feqn
where $ \overset{v}{\Rightarrow}$ denotes the vague convergence of measures on $S^{d-1}.$
The following well-known result is the key to the proof of the next lemma: if $\xi,\eta\in \calr_{d,\alpha,a}$ and $\xi,\eta$ are independent of each other,
then $\nu_{\xi_1+\eta}=\nu_\xi+\nu_\eta$ and $\ms_{\xi+\eta}=\ms_\xi+\ms_\eta.$ We have:
\begin{lemma}
\label{xi-tail}
Let Assumption~\ref{amain} hold. For $t\geq 0,$ let $a_t$ be defined as in \eqref{an}. Then
\item [(a)] $\sum_{\tau_1+1}^{\tau_2}T_i\in \calr_{1,\alpha,a}.$
\item [(b)] $\xi_1\in \calr_{d,\alpha,a}.$
\end{lemma}
\begin{proof}[Proof of Lemma~\ref{xi-tail}]
It is not hard to see that the claim of part~(a) can be formally deduced from that of part~(b).
Thus we will focus on proving the more general claim (b).
\par
First, observe that \eqref{dprop} implies that $T_1u\in \calr_{d,\alpha,a}$ for any $u\in U.$ Let
\beq
H(u,v)=P(\eta_{n+1}=v|\eta_n=u), \qquad u,v\in U,
\feq
be the transition matrix of the Markov chain $\eta.$ Further, define a sub-Markovian kernel $\Theta$ by setting
\beq
\Theta(u,v)=H(u,v)\cdot \one{v\not =u_1}, \qquad u,v\in U.
\feq
Fix any $t>0$ and a Borel set $B\subset S^{d-1},$ and let
\beq
A_n=\bigl\{\|\xi_1\|>ta_n;\,\xi_1/\|\xi_1\|\in B\bigr\},\qquad n\in\nn.
\feq
Then,
\beq
&&
P(\xi_1\in A_n)=\sum_{k=1}^\infty P(\tau_2-\tau_1=k)P(T_1u_1+T_2\eta_2+\ldots +T_k\eta_k \in A_n|\tau_2-\tau_1=k\bigr)=
\\
&&
\quad
\sum_{k=1}^\infty \sum_{v_2\not= u_1}\cdots \sum_{v_k\not= u_1} \Theta(u_1,v_1)\cdots \Theta(v_{k-1},v_k) H(v_k,u_1)
P\bigl(T_1u_1+T_2v_2+\ldots +T_kv_k\in A_n\bigr),
\feq
where we assume that the sums $\sum_{v_2\not= u_1}\cdots \sum_{v_k\not= u_1}$ are empty if $k=1.$
Let
\beq
J_n(v_2,\ldots,v_k)=T_1u_1+T_2v_2+\ldots +T_kv_k.
\feq
Notice that for any $k\in\nn$ and fixed set of vectors $v_2,\ldots,v_k\in U,$ we have
\beq
&&
n \cdot P\bigl(J_n(v_2,\ldots,v_k) \in A_n\bigr)
\leq
n \cdot P\bigl( \|J_n(v_2,\ldots,v_k)\| \geq t a_n \bigr) \leq n P\Bigl( \sum_{j=1}^k T_j \geq t a_n \Bigr)
\\
&&
\qquad
\leq
n k P\bigl(  T_1 \geq t a_n/k \bigr) \leq C t^{-\alpha}k^{1+\alpha}
\feq
for some $C>0.$ Furthermore,
\beq
\lim_{n\to\infty} n \cdot P\bigl(J_n(v_2,\ldots,v_k) \in A_n\bigr) =t^{-\alpha}\Bigl(\ms_{T_1u_1}(B)+\sum_{j=2}^k \ms_{T_1v_j}(B)\Bigr).
\feq
Observe that the spectral radius of the matrix $\Theta$ is strictly less than one and that $\ms_{T_1v_j}(B)$ is uniformly bounded from above
by $\max_{v\in U} \ms_{T_1 v}\bigl(S^{d-1}\bigr).$  Therefore, the dominated convergence theorem implies that the following limit exists and the identity holds:
\beq
&&
\lim_{n\to\infty} n\cdot P(\xi_1\in A_n)
\\
&&
~
=\sum_{k=1}^\infty t^{-\alpha} \sum_{v_2\not= u_1}\cdots \sum_{v_k\not= u_1} \Theta(u_1,v_1)\cdots \Theta(v_{k-1},v_k) H(v_k,u_1)
\Bigl(\ms_{T_1u_1}(B)+\sum_{j=2}^k \ms_{T_1v_j}(B)\Bigr).
\feq
Since the spectral radius of $\Theta$ is strictly less than one, Fubini's theorem implies that
the right-hand side of the above identity defines a measure on $S^{d-1}.$ The proof of the lemma is therefore completed.
\end{proof}
We are now in a position to complete the proof of the limit results stated in parts~(a) and~(b) of Theorem~\ref{clts}.
Suppose first that $\mu = 0.$ It follows from Lemma \ref{xi-tail} and the stable limit theorem for i.i.d. sequences
(see, for instance, Section~1.6 in \cite[p.~75]{htrw-book}) that
\beqn
\label{precon1}
\frac{\sum_{k=1}^{[nt]} \xi_k }{a_n} \Rightarrow S_\alpha(t),\qquad t \in [0,1],
\feqn
where $S_\alpha(t)$ is a homogeneous vector-valued process in $D\bigl(\rr^d\bigr)$ with independent increments and
$S_\alpha(1)$ distributed according to a stable law of index $\alpha.$
Then (similarly to \eqref{general}),  asymptotic equivalence \eqref{malpha} along with the suitable modification
of Theorem~14.4 in \cite[p.~152]{measures} implies
\beq
\frac{\sum_{k=0}^{c([nt])-1} \xi_k }{a_n} \Rightarrow b^{1/\alpha} \cdot S_\alpha(t),\qquad t \in [0,1],
\feq
where $b=\frac{\pi_1}{\cdot E(T_1)}.$ In particular, using $t=1,$
\beqn
\label{precon2}
\frac{\sum_{k=0}^{c(n)-1} \xi_k }{a_n} \Rightarrow b^{1/\alpha} \cdot S_\alpha(1),
\feqn
Recall $r_k$ from \eqref{rk}. Since
\beq
\Bigl\| X_t- \sum_{k=0}^{c(t)-1}\xi_k\Bigr\| \leq r_{c(t)},
\feq
an application of the renewal theorem shows that
\beq
\frac{X_n}{n} \Rightarrow \call_\alpha\qquad \mbox{and hence} \qquad \frac{X_{\lfloor t \rfloor }}{t} \Rightarrow \call_\alpha.
\feq
Since $\bigl\| X_{\lfloor t \rfloor}- X_t\bigr\| \leq 1,$ the proof of part (a)-(i) of Theorem~\ref{clts} is completed.
\end{proof}
$\mbox{}$
\begin{proof} [\textbf{Part (a)-(ii) of Theorem~\ref{clts}}]
For $V \in U$ let $c_v(t)$ be the number of occurrences of $v$ in the set $\{\eta_1,\eta_2,\ldots,\eta_{N_t}\}.$ That is,
\beq
c_v(t)=\sum_{k=1}^{N_t} \one{\eta_k=v},\qquad n\in\nn,\,i\in\cald.
\feq
Notice that $c_{u_1}(t)=c(t),$ where $c(t)$ is introduced in Section~\ref{prelim}. Similarly to \eqref{malpha} we have
\beqn
\label{malpha-g}
 \lim_{t\to\infty}\frac{c_v(t)}{t}= \frac{\pi_v}{E(T_1)},\qquad \as,
\feqn
where $\pi_v$ is the mass that the stationary distribution of the Markov chain $\eta$ puts on $v.$
\par
Define $\tau_v(0)=0$ and $\tau_v(j)=\inf\{k>\tau_v(j-1):\eta_k=v\}$ for $j\in\nn.$ For $v\in U$ and $t\geq 0,$ let
\beq
\witi B_v(t)=\sum_{i=0}^{c_v(t)-1} T_{\tau_v(i)}-c_v(t) \cdot E(T_1).
\feq
Then, the law of iterated logarithm for heavy-tailed i.i.d. sequences (see
Theorems~1.6.6 and 3.9.1 in \cite{htrw-book}) combined with \eqref{malpha} yields
\beqn
\label{decompo4}
\limsup_{t\to\infty} \frac{ \witi B_v(t)}{a_t\cdot (\ln t)^{1/\alpha+\veps}}
=
\left\{
\begin{array}{ll}
0&\mbox{\rm ~if~}~\veps>0,\\
\infty&\mbox{\rm ~if~}~\veps<0
\end{array}
\right.\qquad \mbox{\rm a.\,s.}
\feqn
For $v\in U,$ let
\beq
B_v(t)=\sum_{i=0}^{c_v(t)-1} \bigl(T_{\tau_v(i)}-E(T_1)\bigr)+(t-s_{N_t})\one{\eta_{_{N_t}}=v}.
\feq
Then, \eqref{rw} implies that
\beq
X_t=\sum_{v\in U} vB_v(t)+ \Bigl(\sum_{v\in U} v\cdot c_v(t)\cdot E(T_1)-\mu \cdot t\Bigr).
\feq
Taking into account \eqref{malpha-g}, a standard inversion argument allows one
to deduce from the law of iterated logarithm for $\tau_v(n)$ that
\beqn
\label{esti1}
\limsup_{t\to\infty} \frac{\Bigl\|\sum_{v\in U} v\cdot c_v(t)\cdot E(T_1)-\mu \cdot t\Bigr\| }{\sqrt{t\ln \ln t}} <\infty,\qquad \as
\feqn
Since $a_t \in\calr_\alpha$ with $\alpha\in (1,2),$
\beq
\lim_{t\to\infty} \frac{\sqrt{t\ln \ln t}}{a_t\cdot (\ln t)^{1/\alpha+\veps}}=0.
\feq
Thus \eqref{decompo4} along with \eqref{esti1} yields part (a)-(ii) of Theorem~\ref{clts}, provided
that we are able to show that for any $u,v\in U$ and all $\delta\in (1/(2\alpha),1/\alpha),$
\beqn
\label{remains}
P\Bigl(\bigl(G_{n,v}\cap E_{n,v}\bigr)~\mbox{and}~\bigl(G_{n,u}\cap E_{n,u}\bigr) \,\io\Bigr)=0,
\feqn
where the events $G_{n,v}$ and $E_{n,v}$ are defined for $n\in\nn$ and $v\in U$ as follows.
For $n\in\nn$ let $\gamma_n=2n\cdot \max_{v\in V} \pi_v.$ Let
\beq
E_{n,v}=\Bigl\{ \max_{1\leq m\leq \gamma_n} \Bigl|\sum_{i=0}^{m-1}T_{\tau_v(i)}-m \cdot E(T_1) \Bigr|>a_n\cdot (\ln n)^\delta\Bigr\}
\quad
\mbox{and}
\quad
G_{n,v}=\{ c_v(n)>\gamma_n\}.
\feq
We then have:
\beq
&&
P\Bigl(\bigl(G_{n,v}\cap E_{n,v}\bigr)~\mbox{and}~\bigl(G_{n,u}\cap E_{n,u}\bigr)\Bigr)
\\
&&
\qquad
\leq
P\bigl(E_{n,v}\bigcap E_{n,u}\bigr) +P\bigl(c_v(n)>\gamma_n\bigr)+P\bigl(c_u(n)>\gamma_n\bigr)
\\
&&
\qquad
=
P(E_{n,v})\cdot P(E_{n,u}) +P\bigl(c_v(n)>\gamma_n\bigr)+P\bigl(c_u(n)>\gamma_n\bigr).
\feq
It follows from the large deviation principle for $c_v(n)/n$ that $P\bigl(c_v(n)>\gamma_n\bigr)<K_ve^{-n\lambda_v}$
for some $K_v>0$ and $\lambda_v>0.$ Furthermore, for any $A>0$ and $k_n=[A^n],$ we have
$P(E_{k_n,v})\leq Cn^{-\beta}$ for some constants $\beta>1/2$ and $C>0$
(see \cite[p.~177]{htrw-book}; here we exploit the constraint $2\alpha\delta>1$). The Borel-Cantelli lemma implies
then that $P\bigl(E_{k_n,v}\bigcap E_{k_n,u}~\io\bigr)=0.$ Since for any $n\in\nn$ there is a unique $j(n)\in\nn$ such that
$k_{j(n)}\leq n <k_{j(n)+1},$ and $\lim_{k\to\infty} \frac{a_{k+1}(\ln a_{k+1})^\delta}{a_k(\ln a_k)^\delta}=1,$
this yields \eqref{remains}. The proof of part (a)-(ii) of Theorem~\ref{clts} is therefore completed.
\end{proof}
\subsection{Proof of Theorem~\ref{clt1}}
\label{proof-clt1}
Define two families of processes, $(B_n)_{n\in\nn}$ and $(C_n)_{n\in\nn}$ in $D(\rr),$ by setting
\beqn
\label{precon}
B_n(t)= \frac{\sum_{k=1}^{[nt]} \xi_k }{a_n}\quad \mbox{and} \quad C_n(t)= \frac{s_{\tau_{[nt]}}}{a_n} ,\qquad t \in [0,1].
\feqn
Lemma~\ref{xi-tail} combined with \cite[Theorem~1.1]{character} implies that $(\xi_1,s_{\tau_2}-s_{\tau_1})\in \calr_{d+1,\alpha,a},$ and hence
\beqn
\label{precon4}
(B_n,C_n) \Rightarrow (S_\alpha,U_\alpha),
\feqn
where $S_\alpha$ and $U_\alpha$ are homogeneous process with independent increments in $D\bigl(\rr^d\bigr)$ and $D(\rr),$ respectively,
such that $S_\alpha(1)$ and $U_\alpha(1)$ have (multivariate in the former case) stable distributions of index $\alpha.$
\par
Let $U_n^{-1}$ and $C_n^{-1}$ denote the inverse processes of $U_n$ and $C_n,$ respectively.
One can define $C_n^{-1}$ explicitly as follows:
\beqn
\label{ex}
C_n^{-1}(t) = n^{-1}c(a_n t),\qquad t\in [0,1].
\feqn
Then the same argument as in \cite[pp.~380-381]{limits} shows that (alternatively, one can use the result of \cite{inverse}):
\beq
\bigl(B_n,C_n^{-1}\bigr) \Rightarrow (S_\alpha, U_\alpha^{-1})
\feq
in $D\bigl(\rr^{d+1}\bigr).$ This along with \eqref{ex} implies (see, for instance, \cite[p.~151]{measures}) that
\beq
\frac{\sum_{i=1}^{c(a_n)-1}\xi_i}{a_n} \Rightarrow \call_\alpha,
\feq
where
\beqn
\label{specific}
\call_\alpha:=S_\alpha(U_\alpha^{-1})(1).
\feqn
Passing to the subsequence $m_n=\lfloor a_n^{-1} \rfloor $ and using basic properties
of regularly varying functions, we obtain
\beqn
\label{prec4}
\frac{\sum_{i=0}^{c(n)-1}\xi_i}{n} \Rightarrow \call_\alpha.
\feqn
To conclude the proof of the theorem one can use verbatim the argument
along the lines following \eqref{precon2} in the concluding paragraph of the above proof of part~(a)-(i) of
Theorem~\ref{clts}. Namely, taking into account the inequality
\beq
\Bigl\| X_t- \sum_{k=0}^{c(t)-1}\xi_k\Bigr\| \leq r_{c(t)}
\feq
and using the renewal theorem which ensures the weak convergence of $r_{c(t)}$ to a proper random variable,
\eqref{prec4} yields that $\frac{X_t}{t} \Rightarrow \call_\alpha.$
The proof of Theorem~\ref{clt1} is completed.
\qed

\providecommand{\MRhref}[2]{%
\href{http://www.ams.org/mathscinet-get$item?mr=#1}{#2}} \providecommand{\href}[2]{#2}

\end{document}